\documentclass[a4paper]{amsart}

 \usepackage[dvips]{graphics,color}
 \usepackage{amsmath}
\usepackage{amsfonts}
 \usepackage{amssymb}
 \usepackage{amsthm}

 \newtheorem{lemma}{Lemma}
 \newtheorem{theorem}{Theorem}
 \newtheorem{proposition}{Proposition}
 \newtheorem{corollary}{Corollary}

\newtheorem{example}{Example}

\begin{document}

\title{ON COMPLEX MULTIPLICATIVE INTEGRATION}

\author{Agamirza Bashirov}
\address{Department of Mathematics, Eastern Mediterranean University, Gazimagusa, North Cyprus}
\email{agamirza.bashirov@emu.edu.tr},
        
\author{Mustafa Riza}
\address{Department of Mathematics, Eastern Mediterranean University, Gazimagusa, North Cyprus}
\email{mustafa.riza@emu.edu.tr}



\date{\today}

\begin{abstract}In the present paper we extend the multiplicative integral to complex-valued functions of complex variable. The main difficulty in this way, that is the multi-valued nature of the complex logarithm, is avoided by division of the interval of integration to a finite number of local intervals, in each of which the complex logarithm can be localized in one of its branches. Interestingly, the complex multiplicative integral became a multi-valued function. Some basic properties of this integral are considered. In particular, it is proved that this integral and the complex multiplicative derivative bonded in a kind of fundamental theorem.
\end{abstract}

\maketitle


\section{Introduction}

Earlier in 1938 Volterra and Hostinski \cite{V} invented the bigeometric calculus. This invention was rediscovered later in 1972 by Grossman and Katz  \cite{GK}, who proposed two significant alternative calculi to the classical calculus of Newton and Leibnitz, namely the multiplicative and bigeometric calculi.  These pioneering works initiated numerous studies on multiplicative and bigeometric calculus. In the literature sometimes bigeometric calculus is also referred as Volterra, proportional, or product calculus. Bigeometric calculus was pushed forward by the contributions of Grossman  \cite{G}, Cordova-Lepe \cite{C1, C2}, Slavik \cite{SL} etc. Furthermore elements of stochastic integration of biometric nature are introduced in Karandikar  \cite{K}  and Daletskii and Teterina \cite{DT}. On the other hand, multiplicative calculus and its applications were promoted in Bashirov et al. \cite{B2, B3} and Stanley \cite{ST}. Moreover, Riza et al.  \cite{ROK} and M\i{}s\i{}rl\i{} and G\"{u}refe \cite{MG} used multiplicative calculus for advancement of numerical methods; Florack and van Assen  \cite{FA} applied multiplicative calculus to biomedical image analysis, and Bashirov and Bashirova \cite{B1}  used multiplicative calculus for the derivation of a mathematical model of literary texts etc.

The basic difference between different versions of calculus is that they present calculus with reference to different basic functions. In the case of Newtonian calculus, this reference function is linear. Therefore, the statements and proofs of facts, which are not perfectly described in terms of linear functions become complicated. For example, Newtonian calculus is suitable for Taylor series, but not for Fourier series. At the same time, the reference function of multiplicative calculus is exponential. This makes the study of exponent related problems (for example, growth) suitable in multiplicative calculus. In particular, complex Fourier series, expressed in terms of exponent, are seen to be suitable for multiplicative calculus.

Nevertheless, the capacity of real multiplicative calculus is restricted within the class of positive functions of real variable and, hence, does not accept sine and cosine functions. Therefore, studying Fourier series by means of real multiplicative calculus is not possible. This suggests the creation of complex multiplicative calculus.

A proper complex multiplicative differentiation was prompted in Bashirov and Riza  \cite{B4}, where it was also demonstrated that the complex multiplicative differentiation accepts the functions with positive and as well as negative values whenever they are nowhere-vanishing. This point is unlike to real multiplicative differentiation and very important since the terms of complex Fourier series are also nowhere-vanishing. Continuing this study with the aim of further application to Fourier series, in this paper we deal with complex multiplicative integration by taking into consideration all branches of the complex logarithm.

Many books on complex analysis and calculus are available.  We refer to Ahlfors \cite {AH}, Greene and Krantz \cite{GR} and Sarason \cite{S}, which are used during this study.

The paper is organized in the following way. In Section 2, we briefly review basic points of multiplicative differentiation and line multiplicative integrals. The basic difficulty for a proper definition of the complex multiplicative integral is a multi-valued nature of the complex logarithm. In Section 3 complex multiplicative integral is defined locally, that allows to work with only one branch of the complex logarithm. Next, in Section 4 the multiplicative complex integral is defined in general form. Finally, in Section 5 we study the properties of complex multiplicative integral.

One major remark about the notation is that the multiplicative versions of the concepts of ordinary calculus are called as *concepts, for example, a *derivative means a multiplicative derivative. We denote by $\mathbb{R}$ and $\mathbb{C}$ the fields of real and complex numbers, respectively. $\mathrm{Arg}\,z$  is the principal value of $\arg z$, noticing that $-\pi <\mathrm{Arg}\,z\le \pi$. Always $\ln x$ refers to the natural logarithm of the real number $x>0$ whereas $\log z$ to the same of the complex number  $z\not= 0$. By $\mathrm{Log}\,z$, we denote the value at  $z$ of the principal branch of the complex logarithm, i.e., $\mathrm{Log}\,z=\ln |z|+i\mathrm{Arg}\,z$, where $i$ denotes the imaginary unit and $|z|$ the modulus of $z$.


\section{Preliminaries}

The *derivative $f^*(x)$ of a purely positive or purely negative differentiable function $f$ of a real variable is defined as the limit
 \begin{equation}
 \label{1}
 f^*(x)=\lim _{h\to 0}(f(x+h)/f(x))^{1/h},
 \end{equation}
showing how many times $|f(x)|$ changes at $x$. It differs from the derivative $f'(x)$, which shows to the rate of change of $f$ at $x$. These two derivative concepts are related to each other by the formula
 \begin{equation}
 \label{2}
 f^*(x)=e^{(\ln |f(x)|)'}=e^{\frac{f(x)'}{f(x)}}.
 \end{equation}
The appropriateness of the *derivative, especially, in modeling growth related processes has been demonstrated in various papers, for example, \cite{B1, B2, B3, FA, MG, ROK}.

One can observe that the limit in (\ref{1}) can not be applied to differentiable functions with values changing the sign. The reason is that such a function certainly has zeros. This lack of integrity is removed by complex *derivative.

Following to Bashirov and Riza \cite{B4}, let $f$ be nowhere-vanishing differentiable complex function on an open set $D$ in $\mathbb{C}$. To extend formulae (\ref{1})--(\ref{2}) to the complex case, note that in general a branch of $\log f$ may not exist. Even if it exist, it can not be represented as a composition of a branch of $\log $-function and $f$. These rigors can be avoided locally since for a sufficiently small neighborhood $U\subseteq D$ of the point $z\in D$, the branches of $\log f$ on $U$ exist, they are composition of branches of the $\log $-function and the restriction of $f$ to $U$, and the $\log $-differentiation formula $(\log f)'=f'/f$ is valid for $\log f$ on $U$ (see Sarason \cite{S}). Taking this into consideration, the *derivative of $f$ can be defined just as
 \begin{equation}
 \label{3}
 f^*(z)=e^{(\ln f(z))'}=e^{\frac{f(z)'}{f(z)}},
 \end{equation}
noticing that it is independent of the branches of $\log $-function. In \cite{B4}, it is proved that if $z=x+iy$ and
 \begin{equation}
 \label{4}
 R(z)=R(x,y)=|f(z)|\ \ \text{and}\ \ \Theta (z)=\Theta (x,y)=\mathrm{Arg } f(z),
 \end{equation}
then
 \begin{equation}
 \label{5}
 \left\{ \begin{array}{l}
 |f^*(z)|=R^*_x(z)=[e^{\Theta }]^*_y(z),\\
 \arg f^*(z)=\Theta '_x(z)+2\pi n=-[\ln R]'_y(z)+2\pi n,\ n=0,\pm 1,\pm 2,\ldots ,
 \end{array}\right.
 \end{equation}
where $\Theta '_x$ and $[\ln R]'_y$ are partial derivatives as well as $R^*_x$ and $[e^{\Theta }]^*_y$ are partial *derivatives of the real-valued functions of two real variables.

We say that a complex-valued function $f$ of complex variable is \textit{*differentiable} at $z\in \mathbb{C}$ if it is differentiable at $z$ and $f(z)\not= 0$. We also say that $f$ is \textit{*holomorphic} or \textit{*analytic} on an open connected set $D$ if $f^*(z)$ exists for every $z\in D$.

The following examples demonstrate some features of complex *differentiation.
 \begin{example}
 \label{E1}
{\rm The function $f(z)=e^{cz}$, $z\in \mathbb{C}$, where $c=\mathrm{const}\in \mathbb{C}$, is an entire function and
its *derivative
 \[
 f^*(z)=e^{f'(z)/f(z)}=e^{ce^{cz}/e^{cz}}=e^c,\ z\in \mathbb{C},
 \]
is again an entire function, taking identically the nonzero value $e^c$. Thus, in complex *calculus $f(z)=e^{cz}$ plays the role of the linear function $g(z)=az$ with $a=e^c$ from Newtonian calculus.}
 \end{example}
 \begin{example}
 \label{E2}
{\rm For another entire function $f(z)=e^{ce^z}$, $z\in \mathbb{C}$, with $c=\mathrm{const}\in \mathbb{C}$, we have
 \[
 f^*(z)=e^{f'(z)/f(z)}=e^{ce^ze^{ce^z}/e^{ce^z}}=e^{ce^z},\ z\in \mathbb{C}.
 \]
Hence, $f$ is a solution of the equation $f^*=f$. Thus, in complex *calculus $f(z)=e^{ce^z}$ plays the role of the exponential function $g(z)=ce^z$ from ordinary calculus.}
 \end{example}
 \begin{example}
 \label{E3}
{\rm The function $f(z)=z$, $z\in \mathbb{C}$, is also entire, but its *derivative
 \[
 f^*(z)=e^{f'(z)/f(z)}=e^{1/z},\ z\in \mathbb{C}\setminus \{ 0\},
 \]
accounts an essential singularity at $z=0$. This is because *differentiation is applicable to functions with the range in $\mathbb{C}\setminus \{ 0\} $. Thus, the *derivative of an entire function may not be entire.}
 \end{example}

In order to develop complex *integration, we need also in line *integrals as well as a fundamental theorem of calculus for them. Following to Bashirov \cite{B}, let $f$ be a positive function of two variables, defined on an open connected set in $\mathbb{R}^2$, and let $C$ be a piecewise smooth curve in the domain of $f$. Take a partition $\mathcal{P}=\{ P_0,\ldots ,P_m\} $ on $C$ and let $(\xi _k, \eta _k)$ be a point on $C$ between $P_{k-1}$ and $P_k$. Denote by $\Delta s_k$ the arclength of $C$ from the point $P_{k-1}$ to $P_k$. Define the integral product
 \[
 P(f,\mathcal{P})=\prod _{k=1}^mf(\xi _k,\eta _k)^{\Delta s_k}.
 \]
The limit of this product when $\max \{ \Delta s_1,\ldots ,\Delta s_m\} \to 0$ independently on selection of the points $(\xi _k,\eta _k)$ will be called a \emph{line *integral of $f$ in $ds$ along $C$}, for which we will use the symbol
 \[
 \int _Cf(x,y)^{ds}.
 \]
The line *integral of $f$ along $C$ exist if $f$ is a positive function and the line integral of $\ln f$ along $C$ exists, and they are related as
 \[
 \int _Cf(x,y)^{ds}=e^{\int _C\ln f(x,y)\,ds}.
 \]

In a similar way, the line *integrals in $dx$ and in $dy$ can be defined and their relation to the respective line integrals can be established in the form
 \begin{equation}
 \label{6}
 \int _Cf(x,y)^{dx}=e^{\int _C\ln f(x,y)\,dx}\ \ \text{and}\ \  \int _Cf(x,y)^{dy}=e^{\int _C\ln f(x,y)\,dy}.
 \end{equation}
Clearly, all three kinds of line *integrals exist if $f$ is a positive continuous function. It is also suitable to denote
 \[
 \int _Cf(x,y)^{dx}g(x,y)^{dy}=\int _Cf(x,y)^{dx}\cdot \int _Cg(x,y)^{dy}.
 \]
In cases when $C$ is a closed curve we write $\oint _C$ instead of $\int _C$.
 \begin{example}
 \label{E4}
{\rm Let $c>0$ and let $C=\{ (x(t),y(t)):a\le t\le b\} $ be a piecewise smooth curve. Then
 \[
 \int _Cc^{dx}=e^{\int _C\ln c\,dx}=e^{(x(b)-x(a))\ln c}=c^{x(b)-x(a)}.
 \] }
 \end{example}

 \begin{theorem}[\textit{Fundamental theorem of calculus for line *integrals}]
 \label{T1}
Let $D\subseteq \mathbb{R}^2$ be an open connected set and let $C=\{ (x(t),y(t)):a\le t\le b\} $ be a piecewise smooth
curve in $D$. Assume that $f$ is a continuously differentiable positive function on $D$. Then
 \[
 \int _Cf^*_x(x,y)^{dx}f^*_y(x,y)^{dy}=\frac{f(x(b),y(b))}{f(x(a),y(a))}.
 \]
 \end{theorem}
 \begin{proof}
See Bashirov \cite{B}.
 \end{proof}


 \section{Complex multiplicative integration (local)}

Let $f$ be a continuous nowhere-vanishing complex-valued function of complex variable and let $z(t)=x(t)+iy(t)$, $a\le
t\le b$, be a complex-valued function of real variable, tracing a piecewise smooth simple curve $C$ in the open
connected domain $D$ of $f$. The complex *integral of $f$ along $C$ will heavily use $\log f$. In order to represent
$\log f$ as the composition of branches of $\log $ and $f$ along the whole curve $C$ we will use a ``method of
localization" from Sarason \cite{S}. In this section we will consider a simple case assuming that the length of the
interval $[a,b]$ is sufficiently small so that all the values of $f(z(t))$ for $a\le t\le b$ fall into an open half
plane bounded by a line through the origin. Under this condition the restriction of $\log f$ to $C$ can be treated as a
composition of the branches of $\log $ and the restriction of $f$ to $C$. Moreover, we can select one of the
multi-values of $\log f(z(a))$ and consider a branch $\mathcal{L}$ of $\log $ so that $\mathcal{L}(f(z(a)))$ equals to
this preassigned value. Thus
 \[
 \log f(z(t))=\mathcal{L}(f(z(t)))+2\pi ni,\ a\le t\le b,\ n=0,\pm 1,\pm 2,\ldots .
 \]

Now, take a partition $\mathcal{P}=\{ z_0,\ldots ,z_m\} $ on $C$ and let $\zeta _k$ be a point on $C$ between $z_{k-1}$
and $z_k$. Denote $\Delta z_k=z_k-z_{k-1}$. Consider the integral product $\prod _{k=1}^me^{\Delta z_k\log f(\zeta
_k)}$. It can be evaluated as
 \begin{align*}
 \prod _{k=1}^me^{\Delta z_k\log f(\zeta _k)}
  &=e^{\sum _{k=1}^m(\mathcal{L}(f(\zeta _k))+2\pi ni)\Delta z_k}\\
  &=e^{2\pi n(z(b)-z(a))i}e^{\sum _{k=1}^m\mathcal{L}(f(\zeta _k))\Delta z_k},\ n=0,\pm 1,\pm 2,\ldots ,
 \end{align*}
showing that $\prod _{k=1}^me^{\Delta z_k\log f(\zeta _k)}$ has more than one value. Let
 \begin{equation}
 \label{7}
 P_0(f,\mathcal{P})=e^{\sum _{k=1}^m\mathcal{L}(f(\zeta _k))\Delta z_k}
 \end{equation}
and
 \begin{equation}
 \label{8}
 P_n(f,\mathcal{P})=e^{2\pi n(z(b)-z(a))i}P_0(f,\mathcal{P}),\ n=0,\pm 1,\pm 2,\ldots .
 \end{equation}
The limit of $P_0(f,\mathcal{P})$ as $\max \{ |\Delta z_1|,\ldots ,|\Delta z_m|\} \to 0$ independently on selection of
the points $\zeta _k$ will be called a \emph{branch value of the complex *integral of $f$ along $C$} and it will be
denoted by $I^*_0(f,C)$. Then \emph{the complex *integral of $f$ along $C$} can be defined as the multiple values
 \begin{equation}
 \label{9}
 I^*_n(f,C)=e^{2\pi n(z(b)-z(a))i}I^*_0(f,C),\ n=0,\pm 1,\pm 2,\ldots ,
 \end{equation}
which will be denoted by
 \[
 \int _Cf(z)^{dz}.
 \]
Note that if $z(b)-z(a)$ is an integer, then all the values of $\int _Cf(z)^{dz}$ equal to $I^*_0(f,C)$, i.e.,
$I^*(f,C)$ become single-valued. If $z(b)-z(a)$ is a rational number in the form $p/q$, where $p$ and $q$ are
irreducible integers with $q>0$, then $\int _Cf(z)^{dz}$ has $q$ distinct values
 \[
 e^{2\pi npi/q}I^*_0(f,C),\ n=0,1,\ldots ,q-1 .
 \]
Generally, $\int _Cf(z)^{dz}$ has countably many distinct values. In case if $z(b)-z(a)$ is a real number, we also have
$|I^*_n(f,C)|=|I^*_0(f,C)|$ for all $n$. Similarly, if $z(b)-z(a)$ is an imaginary number, $\mathrm{Arg}\,
I^*_n(f,C)=\mathrm{Arg}\, I^*_0(f,C)$ for all $n$.

The existence of the complex *integral of $f$ can be reduced to the existence of line *integrals in the following way.
Let $R(z)=|f(z)|$ and $\Theta (z)=\mathrm{Im}\,\mathcal{L}(f(z))$ for $z\in C$. Denote $z=x+iy$ and $\Delta z_k=\Delta
x_k+i\Delta y_k$. Then from (\ref{7}),
 \begin{align*}
 P_0(f,\mathcal{P})
  &=e^{\sum _{k=1}^m\mathcal{L}(f(\zeta _k))\Delta z_k}\\
  &=e^{\sum _{k=1}^m(\ln R(\zeta _k)+i\Theta (\zeta _k))(\Delta x_k+i\Delta y_k)}\\
  &=e^{\sum _{k=1}^m(\ln R(\zeta _k)\Delta x_k-\Theta (\zeta _k)\Delta y_k)+i\sum _{k=1}^m(\Theta (\zeta _k)\Delta x_k+\ln R(\zeta _k)\Delta y_k)}.
 \end{align*}
If the limits of the sums in the last expression exist, then they are line integrals, producing
 \begin{equation}
 \label{10}
 I^*_0(f,C)=e^{\int _C(\ln R(z)\,dx-\Theta (z)\,dy)+i\int _C(\Theta(z)\,dx+\ln R(z)\,dy)}.
 \end{equation}
Additionally,
 \[
 e^{2\pi n(z(b)-z(a))i}=e^{2\pi n(-(y(b)-y(a))+i(x(b)-x(a)))}=e^{-\int _C2\pi n\,dy+i\int _C2\pi n\,dx}.
 \]
By (\ref{9})--(\ref{10}), this implies
 \begin{equation}
 \label{11}
 I^*_n(f,C)=e^{\int _C(\ln R(z)\,dx-(\Theta (z)+2\pi n)\,dy)+i\int _C((\Theta(z)+2\pi n)\,dx+\ln R(z)\,dy)}
 \end{equation}
for $n=0,\pm 1,\pm 2,\ldots $ or, in the multi-valued form,
 \[
 \int _Cf(z)^{dz}=e^{\int _C\log f(z)\, dz},
 \]
in which
 \begin{equation}
 \label{12}
 I^*_0(f,C)=e^{\int _C\mathcal{L}(f(z))\,dz}\ \ \text{and}\ \ I^*_n(f,C)=e^{2\pi n(z(b)-z(a))i}I^*_0(f,C).
 \end{equation}

To write (\ref{11}) in terms of line *integrals, note that by (\ref{6})
 \[
 |I^*_n(f,C)| =e^{\int _C(\ln R(z)\,dx-(\Theta (z)+2\pi n)\,dy)}=\int _CR(z)^{dx}\big( e^{-\Theta (z)-2\pi n}\big) ^{dy}
 \]
and
 \begin{align*}
 \arg I^*_n(f,C)
  & =\int _C((\Theta(z)+2\pi n)\,dx+\ln R(z)\,dy)+2\pi m\\
  & =\ln \int _C\big( e^{\Theta (z)+2\pi n}\big) ^{dx}R(z)^{dy}+2\pi m.
 \end{align*}
Hence,
 \begin{equation}
 \label{13}
 I^*_n(f,C)=\int _CR(z)^{dx}\big( e^{-\Theta (z)-2\pi n}\big) ^{dy}e^{i\ln \int _C\left( e^{\Theta (z)+2\pi n}\right) ^{dx}R(z)^{dy}}
 \end{equation}
for $n=0,\pm 1,\pm 2,\ldots \,$. Thus, the conditions, imposed at the beginning of this section, namely, (a) $f$ is
nowhere-vanishing and continuous on the open connected set $D$, (b) $C$ is piecewise smooth and simple curve in $D$,
and (c) $\{ f(z(t)):a\le t \le b\} $ falls into an open half plane bounded by a line through the origin, guarantee the
existence of $\int _Cf(z)^{dz}$ as multiple values.

The following proposition will be used for justifying the correctness of the definition of the complex *integral for
arbitrary interval $[a,b]$ in the next section.
 \begin{proposition}[\textit{1st multiplicative property, local}]
 \label{P1}
Let $f$ be a nowhere-vanishing continuous function, defined on an open connected set $D$, and let $C=\{
z(t)=x(t)+iy(t):a\le t\le b\} $ be a piecewise smooth simple curve in $D$ with the property that the set $\{
f(z(t)):a\le t\le b\} $ falls into an open half plane bounded by a line through origin. Take any $a<c<b$ and let
$C_1=\{ z(t)=x(t)+iy(t):a\le t\le c\} $ and $C_2=\{ z(t)=x(t)+iy(t):c\le t\le b\} $. Then
 \[
 \int _Cf(z)^{dz}=\int _{C_1}f(z)^{dz}\int _{C_2}f(z)^{dz},
 \]
where the equality is understood in the sense that
 \[
 I^*_n(f,C)=I^*_n(f,C_1)I^*_n(f,C_2)\ \ \text{for all}\ \ n=0,\pm 1,\pm 2,\ldots
 \]
with the same branch $\mathcal{L}$ of $\log $ used for $I_0(f,C)$, $I_0(f,C_1)$ and $I_0(f,C_2)$.
 \end{proposition}
 \begin{proof}
This follows immediately from (\ref{11}) and the respective property of line integrals.
 \end{proof}

Next, we consider a *analog of the fundamental theorem of complex calculus in a local form.
 \begin{proposition}[\textit{Fundamental theorem of complex *calculus, local}]
 \label{P2}
Let $f$ be a nowhere-vanishing *holomorphic function, defined on an open connected set $D$, and let $C=\{
z(t)=x(t)+iy(t):a\le t\le b\} $ be a piecewise smooth simple curve in $D$ with the property that the set $\{
f(z(t)):a\le t\le b\} $ falls into an open half plane bounded by a line through origin. Then
 \[
 \int _Cf^*(z)^{dz}=\{ e^{2\pi n(z(b)-z(a))i}f(z(b))/f(z(a)): n=0,\pm 1,\pm 2,\ldots \} .
 \]
 \end{proposition}
 \begin{proof}
From (\ref{13}),
 \[
 \int _Cf^*(z)^{dz}=\int _C|f^*(z)|^{dx}\big( e^{-\arg f^*(z)}\big) ^{dy}\ e^{i\ln \int _C\left( e^{\arg f^*(z)}\right) ^{dx}|f^*(z)|^{dy}}.
 \]
Using (\ref{5}),
 \begin{align*}
 I_n^*(f^*,C)
  &=\int _CR^*_x(z)^{dx}\big( e^{[\ln R]'_y(z)-2\pi n}\big) ^{dy}\ e^{i\ln \int _C\left( e^{\Theta '_x(z)+2\pi n}\right) ^{dx}
    \left[ e^\Theta \right] ^*_y(z)^{dy}}\\
  &=\int _CR^*_x(z)^{dx}R^*_y(z)^{dy}\ e^{i\ln \int _C\left[ e^\Theta \right] ^*_x(z)^{dx}\left[ e^{\Theta }\right] ^*_y(z)^{dy}}\\
  &\ \ \ \times  \int _C\big( e^{-2\pi n}\big) ^{dy}\ e^{i\ln \int _C(e^{2\pi n})^{dx}}.
 \end{align*}
By Theorem \ref{T1},
 \[
 \int _CR^*_x(z)^{dx}R^*_y(z)^{dy}\ e^{i\ln \int _C\left[ e^\Theta \right] ^*_x(z)^{dx}\left[ e^{\Theta }\right] ^*_y(z)^{dy}}
 =\frac{R(z(b))e^{i\Theta (z(b))}}{R(z(a))e^{i\Theta (z(a))}}=\frac{f(z(b))}{f(z(a))},
 \]
and, by Example \ref{E4},
 \[
 \int _C\big( e^{-2\pi n}\big) ^{dy}\ e^{i\ln \int _C(e^{2\pi n})^{dx}}=e^{2\pi n(-(y(b)+y(a))+i(x(b)-x(a)))}=e^{2\pi n(z(b)-z(a))i}.
 \]
Thus, the proposition is proved.
 \end{proof}


 \section{Complex multiplicative integration (general)}

The following lemma is crucial for a general definition of complex *integral.
 \begin{lemma}
 \label{L1}
Let $f$ be a continuous nowhere-vanishing function, defined on an open connected set $D$, and let $C=\{
z(t)=x(t)+iy(t):a\le t\le b\} $ be a piecewise smooth curve in $D$. Then there exists a partition $\mathcal{P}=\{
t_0,t_1,\ldots ,t_m\} $ of $[a,b]$ such that each of the sets $\{ f(z(t)):t_{k-1}\le t\le t_k\} $, $k=1,\ldots ,m$,
falls into an open half plane bounded by a line through origin.
 \end{lemma}
 \begin{proof}
To every $t\in [a,b]$, consider $\theta _t=\mathrm{Arg}\,f(z(t))$ and the line $L_t$ formed by the rays $\theta =\theta
_t+\pi/2$ and $\theta =\theta _t-\pi/2$. Since $f$ is continuous and nowhere-vanishing, there is an interval
$(t-\varepsilon _t ,t+\varepsilon _t )\subseteq [a,b]$ such that the set
 \[
 \{ f(z(s)):t-\varepsilon _t <s<t+\varepsilon _t \}
 \]
falls into one of the open half planes bounded by $L_t$ if $t\in (a,b)$. In the case of $t=a$ such an interval can be
selected in the form $[a,a+\varepsilon _a )$ and in the case $t=b$ as $(b-\varepsilon _b,b]$. The collection of all such
intervals forms an open cover of the compact subspace $[a,b]$ of $\mathbb{R}$. Therefore, there is a finite number of
them covering $[a,b]$. Writing the end points of these intervals in an increasing order $a=t_0<t_1<\cdots <t_m=b$
produces a required partition.
 \end{proof}

This lemma determines a way for definition of $\int _Cf(z)^{dz}$ in the general case. Assume again that $f$ is a
continuous nowhere-vanishing complex-valued function of complex variable and $z(t)=x(t)+iy(t)$, $a\le t\le b$, is a
complex-valued function of real variable, tracing a piecewise smooth simple curve $C$ in the open connected domain $D$
of $f$. Let $\mathcal{P}=\{ t_0,t_1,\ldots ,t_m\} $ be a partition of $[a,b]$ from Lemma \ref{L1} and let $C_k=\{
z(t):t_{k-1}\le t\le t_k\} $. Choose any branch $\mathcal{L}_1$ of $\log $ and consider $\int _{C_1}f(z)^{dz}$ as
defined in the previous section. Then select a branch $\mathcal{L}_2$ of $\log $ with
$\mathcal{L}_2(f(z(t_1)))=\mathcal{L}_1(f(z(t_1)))$ and consider $\int _{C_2}f(z)^{dz}$. Next, select a branch
$\mathcal{L}_3$ of $\log $ with $\mathcal{L}_3(f(z(t_2)))=\mathcal{L}_2(f(z(t_2)))$ and consider $\int
_{C_3}f(z)^{dz}$, etc. In this process the selection of the starting branch $\mathcal{L}_1$ is free, but the other
branches $\mathcal{L}_2,\ldots ,\mathcal{L}_m$ are selected accordingly to construct a continuous single-valued
function $g$ on $[a,b]$ so that the value of $g$ at fixed $t\in [a,b]$ equals to one of the branch values of $\log
f(z(t))$. Now, following to (\ref{12}), the \emph{complex *integral of $f$ over $C$}, that will again be denoted by
$\int _Cf(z)^{dz}$, can be defined as the multiple values
 \[
 I^*_n(f,C)=\prod _{k=1}^me^{2\pi n(z(t_k)-z(t_{k-1}))i+\int _{C_k}\mathcal{L}_k(f(z))\,dz},\ n=0,\pm 1,\pm 2,\ldots ,
 \]
or
 \begin{equation}
 \label{14}
 I^*_n(f,C)=e^{2\pi n(z(b)-z(a))i}e^{\sum _{k=1}^m\int _{C_k}\mathcal{L}_k(f(z))\,dz},\ n=0,\pm 1,\pm 2,\ldots .
 \end{equation}

This definition is independent on the selection of the partition $\mathcal{P}$ of $[a,b]$. Indeed, if $\mathcal{Q}$ is
another partition, being a refinement of the previous one, then the piece $C_k$ from $z(t_{k-1})$ to $z(t_k)$ of the
curve $C$ became departed into smaller non-overlapping pieces $C_{ki}$, $i=1,\ldots ,l_k$, each over the partition
intervals of $\mathcal{Q}$ falling into $[t_{k-1},t_k]$. Since the range of $f$ over $C_k$ falls into an open half
plane bounded by a line through origin, the range of $f$ over each $C_{ki}$ falls into the same half plane. Therefore,
by Proposition \ref{P1},
 \[
 I^*_n(f,C_k)=\prod _{i=1}^{l_k}I^*_n(f,C_{ki})
 \]
with the same branch $\mathcal{L}_k$ of $\log $ used for all $I_0(f,C_k)$ and $I_0(f,C_{k1}),\ldots ,I_0(f,C_{kl_k})$.
Then
 \[
 I^*_n(f,C_k)=\prod _{k=1}^{m}I^*_n(f,C_k))=\prod _{k=1}^{m}\prod _{i=1}^{l_k}I^*_n(f,C_{ki}),
 \]
i.e., both $\mathcal{P}$ and $\mathcal{Q}$ produce the same multiple values. In case if $\mathcal{P}$ and $\mathcal{Q}$
are two arbitrary partitions of $[a,b]$ from Lemma \ref{L1}, we can compare the integral for selections $\mathcal{P}$
and $\mathcal{Q}$ with the same for their refinement $\mathcal{P}\cup \mathcal{Q}$ and deduce that $I^*_n(f,C)$ is
independent on selection of $\mathcal{P}$ and $\mathcal{Q}$. Actually the described method can be seen as a kind of {\em gluing} method, where we match the values for the partitions  $\mathcal P$ and $\mathcal Q$. 


 \section{Properties of Complex Multiplicative Integrals}

 \begin{theorem}[\textit{1st multiplicative property}]
 \label{T2}
Let $f$ be a nowhere-vanishing continuous function, defined on an open connected set $D$, and let $C=\{
z(t)=x(t)+iy(t):a\le t\le b\} $ be a piecewise smooth simple curve in $D$. Take any $a<c<b$ and let $C_1=\{
z(t)=x(t)+iy(t):a\le t\le c\} $ and $C_2=\{ z(t)=x(t)+iy(t):c\le t\le b\} $. Then
 \[
 \int _Cf(z)^{dz}=\int _{C_1}f(z)^{dz}\int _{C_2}f(z)^{dz}
 \]
in the sense that $I^*_n(f,C)=I^*_n(f,C_1)I^*_n(f,C_2)$ for all $n=0,\pm 1,\pm 2,\ldots $, where
$\mathcal{L}_{01}(f(z(a)))=\mathcal{L}_{11}(f(z(a)))$ and $\mathcal{L}_{1m_1}(f(z(c)))=\mathcal{L}_{21}(f(z(c)))$ if
$\mathcal{L}_{01},\ldots ,\mathcal{L}_{0m}$, $\mathcal{L}_{11},\ldots ,\mathcal{L}_{1m_1}$ and $\mathcal{L}_{21},\ldots
,\mathcal{L}_{2m_2}$ are the sequences of branches of $\log $ used in definition of $I^*_0(f,C)$, $I^*_0(f,C_1)$ and
$I^*_0(f,C_2)$, respectively.
 \end{theorem}
 \begin{proof}
This follows from the the way of definition of complex *integral for general interval $[a,b]$ and its independence on
the selection of partition $\mathcal{P}$ from Lemma \ref{L1}.
 \end{proof}
 \begin{theorem}[\textit{2nd multiplicative property}]
 \label{T3}
Let $f$ and $g$ be nowhere-vanishing continuous functions, defined on an open connected set $D$, and let $C=\{
z(t)=x(t)+iy(t):a\le t\le b\} $ be a piecewise smooth simple curve in $D$. Then
 \[
 \int _C(f(z)g(z))^{dz}=\int _Cf(z)^{dz}\int _Cg(z)^{dz}
 \]
as a set equality, where the product of the sets $A$ and $B$ is treated as $AB=\{ ab: a\in A,\ b\in B\} $.
 \end{theorem}
 \begin{proof}
This follows from (\ref{14}) and the set equality $\log (z_1z_2)=\log z_1+\log z_2$, where the sum of the sets $A$ and
$B$ is treated as $A+B=\{ a+b: a\in A,\ b\in B\} $.
 \end{proof}
 \begin{theorem}[\textit{Division property}]
 \label{T4}
Let $f$ and $g$ be nowhere-vanishing continuous functions, defined on an open connected set $D$, and let $C=\{
z(t)=x(t)+iy(t):a\le t\le b\} $ be a piecewise smooth simple curve in $D$. Then
 \[
 \int _C(f(z)/g(z))^{dz}=\int _Cf(z)^{dz}\big/ \int _Cg(z)^{dz}
 \]
as a set equality, where the ratio of the sets $A$ and $B$ is treated as $A/B=\{ a/b: a\in A,\ b\in B\} $.
 \end{theorem}
 \begin{proof}
This follows from (\ref{14}) and the set equality  $\log (z_1/z_2)=\log z_1-\log z_2$, where the difference of the sets
$A$ and $B$ is treated as $A-B=\{ a-b: a\in A,\ b\in B\} $.
 \end{proof}
 \begin{theorem}[\textit{Reversing the curve}]
 \label{T5}
Let $f$ be a nowhere-vanishing continuous functions, defined on an open connected set $D$, let $C=\{
z(t)=x(t)+iy(t):a\le t\le b\} $ be a piecewise smooth simple curve in $D$ and let $-C$ be the curve $C$ with opposite
orientation. Then
 \[
 \int _Cf(z)^{dz}=\Big( \int _{-C}f(z)^{dz}\Big) ^{-1}
 \]
in the sense that $I^*_n(f,C)=I^*_n(f,-C)^{-1}$ for all $n=0,\pm 1,\pm 2,\ldots \,$, where
$\mathcal{L}_{11}(f(z(a)))=\mathcal{L}_{2m_2}(f(z(a)))$ or $\mathcal{L}_{1m_1}(f(z(b)))=\mathcal{L}_{21}(f(z(b)))$ if
$\mathcal{L}_{11},\ldots ,\mathcal{L}_{1m_1}$ and $\mathcal{L}_{21},\ldots ,\mathcal{L}_{2m_2}$ are the sequences of
branches of $\log $ used in definition of $I^*_0(f,C)$ and $I^*_0(f,-C)$, respectively.
 \end{theorem}
 \begin{proof}
This follows from (\ref{14}).
 \end{proof}
 \begin{theorem}[\textit{Raising to a natural power}]
 \label{T6}
Let $f$ be a nowhere-vanishing continuous functions, defined on an open connected set $D$ and let $C=\{
z(t)=x(t)+iy(t):a\le t\le b\} $ be a piecewise smooth simple curve in $D$. Then for $n=0,1,2,\ldots \,$,
 \[
 \Big( \int _Cf(z)^{dz}\Big) ^n\subseteq \int _C(f(z)^n)^{dz}.
 \]
 \end{theorem}
 \begin{proof}
This follows from multiple application of Theorem \ref{T3} and the fact that $A^n\subseteq AA\cdots A$ ($n$ times),
where the set $A^n$ is treated as $A^n=\{ a^n: a\in A\} $, but $AA\cdots A(n\ \text{times})=\{ a_1a_2\cdots a_n:a_i\in
A, i=1,\ldots ,n\} $.
 \end{proof}
 \begin{theorem}[\textit{Fundamental theorem of calculus for complex *integrals}]
 \label{T7}
Let $f$ be a nowhere-vanishing *holomorphic function, defined on an open connected set $D$, and let $C=\{
z(t)=x(t)+iy(t):a\le t\le b\} $ be a piecewise smooth simple curve in $D$. Then
 \begin{equation}
 \label{15}
 \int _Cf^*(z)^{dz}=\{ e^{2\pi n(z(b)-z(a))i}f(z(b))/f(z(a)),\ n=0,\pm 1,\pm 2,\ldots \} .
 \end{equation}
 \end{theorem}
 \begin{proof}
Let $\mathcal{P}=\{ t_0,t_1,\ldots ,t_m\} $ be a partition from Lemma \ref{L1} and let $C_k=\{ z(t):t_{k-1}\le t\le
t_k\} $. Then
 \[
 \int _Cf(z)^{dz}=\int _{C_1}f(z)^{dz}\cdots \int _{C_m}f(z)^{dz}.
 \]
Hence, by Proposition \ref{P2},
 \begin{align*}
 I^*_n(f,C)
  & =e^{2\pi n(z(t_1)-z(t_0))i+\cdots +2\pi  n(z(t_m)-z(t_{m-1}))i}\frac{f(z(t_1))\cdots f(z(t_m))}{f(z(t_0))\cdots  f(z(t_{m-1}))}\\
  & =e^{2\pi n(z(b)-z(a))i}\frac{f(z(b))}{f(z(a))}.
 \end{align*}
This proves the theorem.
 \end{proof}

This theorem demonstrates that $\int _Cf^*(z)^{dz}$ is independent on the shape of the piecewise smooth curve $C$, but
depends on its initial point $z(a)=x(a)+iy(a)$ and end point $z(b)=x(b)+iy(b)$ on the curve $C$. Therefore, this
integral can be denoted by
 \[
 \int _{z(a)}^{z(b)}f^*(z)^{dz}.
 \]
 \begin{corollary}
 \label{C1}
Let $f$ be a nowhere-vanishing *holomorphic function, defined on an open connected set $D$, and let $C=\{
z(t)=x(t)+iy(t):a\le t\le b\} $ be a piecewise smooth simple closed curve in $D$. Then
 \begin{equation}
 \label{16}
 \oint _Cf^*(z)^{dz}=1.
 \end{equation}
 \end{corollary}
 \begin{proof}
Simply, write $z(a)=z(b)$ in (\ref{15}).
 \end{proof}

Note that in (\ref{16}) all the values of $\oint _Cf^*(z)^{dz}$ are equal to 1, i.e., $\oint _Cf^*(z)^{dz}$ becomes
single-valued.
 \begin{example}
 \label{E5}
{\rm By Example \ref{E1}, the function $f(z)=e^{cz}$, $z\in \mathbb{C}$, where $c=\mathrm{const}\in \mathbb{C}$, has the
*derivative $f^*(z)=e^c$. Respectively,
 \[
 \int _C(e^c)^{dz}=e^{2\pi n(z(b)-z(a))i}e^{c(z(b)-z(a))}=e^{(z(b)-z(a))(c+2\pi ni)},
 \]
where $C=\{ z(t):a\le t\le b\} $ is a piecewise smooth curve. }
 \end{example}
 \begin{example}
 \label{E6}
{\rm By Example \ref{E2}, the function $f(z)=e^{ce^z}$, $z\in \mathbb{C}$, where $c=\mathrm{const}\in \mathbb{C}$, has
the
*derivative $f^*(z)=f(z)$. Respectively,
 \[
 \int _C\big( e^{ce^z}\big) ^{dz}=e^{2\pi n(z(b)-z(a))i}e^{c\left( e^{z(b)}-e^{z(a)}\right) },
 \]
where again $C=\{ z(t):a\le t\le b\} $ is a piecewise smooth curve.
 }
 \end{example}
 \begin{example}
 \label{E66}
{\rm The analog of the integral
 \[
 \oint _{|z|=1}\frac{dz}{z}=2\pi i
 \]
in complex *calculus is
 \[
 \oint _{|z|=1}\big( e^{1/z}\big) ^{dz}.
 \]
Assuming that the orientation on the unit circle $|z|=1$ is positive, we informally have
 \[
 \oint _{|z|=1}\big( e^{1/z}\big) ^{dz}=e^{\oint _{|z|=1}\log e^{1/z}dz}=e^{\oint _{|z|=1}\big( \frac{1}{z}+2\pi ni\big)dz}
 =e^{2\pi n(z(b)-z(a))i}e^{2\pi i}=1.
 \]
Formally, we use Example \ref{E3} and calculate the same:
 \[
 \oint _{|z|=1}\big( e^{1/z}\big) ^{dz}=e^{2\pi n(z(b)-z(a))i}\frac{z(b)}{z(a)}=\frac{e^{\pi i}}{e^{-\pi i}}=e^{2\pi i}=1.
 \]
Thus, this example also fits to Corollary \ref{C1}. The main idea of this is that $e^0=e^{2\pi i}=1$ though $2\pi
i\not= 0$. In other words, the discontinuity of the branches of $\log $ on $(-\infty ,0]$, that creates the Cauchy formula $\int
_{|z|=1}\frac{dz}{z}=2\pi i$, appears in a smooth form in complex *calculus because now $\log z$ is replaced by
$e^{\log z}=z$, where the discontinuity of $\log $ is compensated by periodicity of the exponential function.
 }
 \end{example}

\section{Conclusion}

In continuation of \cite{B4}, evidently, the extension of multiplicative calculus to complex valued functions of complex variable eliminates the restriction to positive valued functions caused by multiplicative calculus. The complex multiplicative integral was defined firstly using the "method of localization"  in the sense of \cite{S} so that the values of the function are restricted to one half-plane bounded by a line through the origin, so we can decompose the restriction of $\log f$ on $C$ into the branches of the complex logarithm and the restriction of $f$ to $C$. The general definition of the complex multiplicative integral removes the restriction by  the so-called {\em{gluing} }method by matching the values of the functions at the branch-cuts. Finally, based on this definition the properties of complex multiplicative integrals are given and illustrated by application to certain standard examples.  


 \end{document}